\documentclass[11pt]{article}

\usepackage[T1]{fontenc}
\usepackage[utf8]{inputenc}
\usepackage{lmodern}
\usepackage[letterpaper,margin=1in]{geometry}
\usepackage{amsmath,amssymb,amsthm,mathtools}
\usepackage{booktabs}
\usepackage{aliascnt}
\usepackage{cite}
\numberwithin{equation}{section}
\usepackage[colorlinks=true,linkcolor=blue,citecolor=blue,urlcolor=blue]{hyperref}
\usepackage[nameinlink,noabbrev]{cleveref}
\usepackage{xcolor}
\usepackage[normalem]{ulem}

\newtheorem{theorem}{Theorem}[section]
\newaliascnt{lemma}{theorem}
\newtheorem{lemma}[lemma]{Lemma}
\aliascntresetthe{lemma}
\newaliascnt{proposition}{theorem}
\newtheorem{proposition}[proposition]{Proposition}
\aliascntresetthe{proposition}
\newaliascnt{corollary}{theorem}
\newtheorem{corollary}[corollary]{Corollary}
\aliascntresetthe{corollary}
\newaliascnt{conjecture}{theorem}
\newtheorem{conjecture}[conjecture]{Conjecture}
\aliascntresetthe{conjecture}
\theoremstyle{remark}
\newaliascnt{remark}{theorem}
\newtheorem{remark}[remark]{Remark}
\aliascntresetthe{remark}

\crefname{lemma}{lemma}{lemmas}
\Crefname{lemma}{Lemma}{Lemmas}
\crefname{proposition}{proposition}{propositions}
\Crefname{proposition}{Proposition}{Propositions}
\crefname{corollary}{corollary}{corollaries}
\Crefname{corollary}{Corollary}{Corollaries}
\crefname{conjecture}{conjecture}{conjectures}
\Crefname{conjecture}{Conjecture}{Conjectures}
\crefname{remark}{remark}{remarks}
\Crefname{remark}{Remark}{Remarks}

\newcommand{\R}{\mathbb R}
\newcommand{\C}{\mathbb C}
\newcommand{\E}{\mathbb E}
\newcommand{\eps}{\varepsilon}
\newcommand{\diag}{\operatorname{diag}}
\newcommand{\tr}{\operatorname{tr}}
\newcommand{\Cov}{\operatorname{Cov}}
\newcommand{\maxroot}{\operatorname{maxroot}}
\newcommand{\adj}{\operatorname{adj}}
\newcommand{\Hmom}{\mathsf H}
\DeclarePairedDelimiter{\norm}{\lVert}{\rVert}

\title{Improved Bounds for the Bilu--Linial Conjecture via Spectral Recovery from Mixed Determinantal Polynomials}
\author{Fangfang Lin\footnote{School of Mathematics and Statistics, Fuzhou University. Email: \href{mailto:2500310002@fzu.edu.cn}{2500310002@fzu.edu.cn}.}
\and
Hong Zhou\footnote{Corresponding author. School of Mathematics and Statistics, Fuzhou University. Email: \href{mailto:hong.zhou@fzu.edu.cn}{hong.zhou@fzu.edu.cn}.}}

\date{}

\hypersetup{
  pdfview=FitH,
  pdfstartview={FitH null},
  pdfpagelayout=OneColumn
}

\begin{document}

\maketitle

\begin{abstract}
Bilu and Linial conjectured that every finite $d$-regular graph with $d \geq 2$ admits an edge signing $\sigma$ whose signed adjacency matrix $A_\sigma$ has spectral radius at most $2\sqrt{d-1}$.
We prove that every signing meeting the mixed-root condition $r_{A_\sigma}\leq\sqrt{2(d-1)}$ satisfies
\[
  \rho(A_\sigma) < \frac{3+\sqrt5}{2}\sqrt{d-1},
\]
where $r_{A_\sigma}$ is the largest root of the mixed determinantal polynomial $\chi[A_\sigma,-A_\sigma]$.
The interlacing theorem of Ravichandran and Srivastava guarantees a signing satisfying the mixed-root condition, so our result improves the coefficient $2\sqrt2$ in their two-sided spectral bound.

In the proof, we construct a positive matrix-valued probability measure supported on the roots of $\chi[A_{\sigma},-A_{\sigma}]$. 
The second moment gives a simple matrix inequality $A_{\sigma}^2 + dI \preceq 4r_{A_\sigma}^2I$, which yields a preliminary coefficient $\sqrt{7}$.
Estimates for the fourth moment use information about short walks to obtain the coefficient $(3+\sqrt{5})/2$. 
For triangle-free graphs, the coefficient improves to $\sqrt6$, and for graphs with girth at least five, it improves to $\sqrt{35/6}$.
As a result of independent interest, we extend the construction to $\chi[A_1,\ldots,A_k]$ for Hermitian matrices $A_1,\ldots,A_k$ with zero diagonal, and compute the first two moments explicitly.

Finally, an explicit signing of $K_8$ shows that the mixed-root condition alone cannot guarantee a coefficient below $(4+\sqrt5)/\sqrt6$.
\end{abstract}

\section{Introduction}

Ramanujan graphs are sparse regular graphs with essentially optimal spectral expansion~\cite{LPS88,HLW06}.
They arise in pseudorandomness, rapidly mixing random walks, communication networks~\cite{HLW06}, and error-correcting codes~\cite{SS96}. 
A connected $d$-regular graph is Ramanujan if every nontrivial adjacency eigenvalue satisfies
\[
  |\lambda|\leq2\sqrt{d-1};
\]
the trivial eigenvalue is $d$ (the bipartite case also includes $-d$).
This threshold is the spectral radius of the adjacency operator on the infinite $d$-regular tree~\cite{LPS88,HLW06}.
The Alon--Boppana bound shows that it is asymptotically optimal for families of $d$-regular graphs of increasing size~\cite{Nil91}.

Bilu and Linial proposed an approach to construct large Ramanujan graphs recursively by $2$-lifts~\cite{BL06}. 
An edge signing $\sigma:E\to\{-1,+1\}$ of a graph $G=(V,E)$ determines a $2$-lift whose adjacency spectrum consists of the eigenvalues of $G$, called the \emph{old} eigenvalues, and those of the signed adjacency matrix $A_\sigma$, called the \emph{new} eigenvalues. 
Thus, a $2$-lift of a Ramanujan graph remains Ramanujan if all new eigenvalues belong to $[-2\sqrt{d-1},2\sqrt{d-1}]$. 
Bilu and Linial~\cite{BL06} proposed the following conjecture.

\begin{conjecture}[Bilu--Linial\footnote{Xu~\cite{Xu26a} recently disproved this conjecture for general regular graphs.
A natural extremal problem remains: what is the smallest universal constant $C$ such that every finite $d$-regular graph with $d\geq2$ admits a signing $\sigma$ satisfying $\rho(A_\sigma)\leq C\sqrt{d-1}$?}]
\label{conj:bilu-linial}
Every finite $d$-regular graph with $d\geq2$ has a signing $\sigma$ such that
\[
  \rho(A_\sigma)
  :=\max\{\lambda_{\max}(A_\sigma),-\lambda_{\min}(A_\sigma)\}
  \leq2\sqrt{d-1}.
\]
\end{conjecture}
The $2$-lift approach therefore needs to control both ends of a signed adjacency spectrum. 
Bilu and Linial proved a weaker bound $\rho(A_\sigma)=O(\sqrt{d\log^3d})$ in the same work~\cite{BL06}.

There has been considerable interest in \Cref{conj:bilu-linial}~\cite{LN23,Sri24}.
However, Marcus, Spielman, and Srivastava~\cite{MSS15} obtained a major advance by using interlacing families to obtain the one-sided estimate
\[
  \lambda_{\max}(A_\sigma)\leq2\sqrt{d-1}
\]
for every $d$-regular graph.
For bipartite graphs, every signed adjacency spectrum is symmetric about zero, so this also controls the smallest eigenvalue and proves \Cref{conj:bilu-linial} in the bipartite case.
Repeated $2$-lifts then give bipartite Ramanujan graphs of all degrees.

For general graphs, the two spectral extremes must be controlled for the same signing without this symmetry. 
Ravichandran and Srivastava used mixed determinantal polynomials to obtain\footnote{Recently, Huang~\cite{Hua26} improves the bound \eqref{eq:RS} slightly to $\rho(A_\sigma) \le 4\sqrt{\left\lceil d/2 \right\rceil-1}$ with a combinatorial signing strategy.}
\begin{equation}
  \rho(A_\sigma)\leq2\sqrt{2(d-1)}.
  \label{eq:RS}
\end{equation}

Ravichandran and Srivastava's method first bounds the roots of an auxiliary polynomial and then converts that root bound into a bound on the signed adjacency spectrum~\cite[Theorem~6 and Section~7.2]{RS21}.
We improve the conversion step by extracting matrix moments from the mixed roots.

\subsection{Our contributions}

\subsection*{An improved bound for the Bilu-Linial conjecture}

For a real symmetric zero-diagonal matrix $A$, write
\[
r_A:=\maxroot \chi[A,-A](x),
\]
where $\chi$ is the normalized mixed determinantal polynomial defined in \Cref{sec:mixed-background}. 
Ravichandran and Srivastava~\cite{RS21} showed that the polynomials $\{ \chi[A_{\sigma},-A_{\sigma}] \}_{\sigma}$ form an interlacing family, and their average is a rescaled matching polynomial. 
The Heilmann--Lieb bound on the roots of the matching polynomial~\cite{HL72} together with the interlacing property guarantees the existence of a signing $\sigma$ satisfying
\begin{equation}
  r_{A_\sigma} \leq \sqrt{2(d-1)}.
  \label{eq:mixed-root-condition}
\end{equation}
We call \eqref{eq:mixed-root-condition} the \emph{mixed-root condition}.
The conversion step in~\cite{RS21} then shows that
\begin{equation}
  \rho(A)\leq 2r_{A}
  \label{eq:rho-rA}
\end{equation}
holds for any real symmetric zero-diagonal matrix $A$, which then yields \eqref{eq:RS}. 

Our first result is an improvement to the conversion step.
\begin{theorem}[Two-sided signing bound]
\label{thm:main-signing}
For every finite $d$-regular graph with $d\geq2$,
\[
  \max_{\sigma:\,r_{A_\sigma}^2\leq2(d-1)}\rho(A_\sigma)
  <\frac{3+\sqrt5}{2}\sqrt{d-1}.
\]
In particular, every such graph admits a signing satisfying this bound.
\end{theorem}

The coefficient $(3+\sqrt5)/2\approx2.61803$ improves the coefficient
$2\sqrt2\approx2.82843$ in \eqref{eq:RS}. 

\begin{remark}
The existence bound also holds for every finite simple graph of maximum degree at most $d$.
Indeed, the standard regular-completion theorem embeds such a graph as an induced subgraph of a finite simple $d$-regular graph; see Erd\H{o}s and Kelly~\cite[p.~65]{Erd67} for this result of K\"onig. 
Restricting a signing to an induced subgraph cannot increase the operator norm of its signed adjacency matrix.
\end{remark}

In the proof of \Cref{thm:main-signing}, we found that short cycles play an important role in the spectral radius estimation, which leads to stronger bounds when short cycles are absent.

\begin{corollary}[Bounds for graphs without short cycles]
\label{cor:short-cycles}
Let $A$ be a signed adjacency matrix of a finite $d$-regular graph, $d\geq2$, satisfying the mixed-root condition \eqref{eq:mixed-root-condition}. 
If the graph is triangle-free, then
\begin{equation}
  \rho(A)^2\leq
  \frac{5d-6+\sqrt{49d^2-92d+36}}2<6(d-1).
  \label{eq:triangle-free-bound}
\end{equation}
If its girth is at least five, then
\begin{equation}
  \rho(A)^2
  \leq
  \frac{d(7d-8)(5d-7)}{6d^2-9d+1}
  <
  \frac{35}{6}(d-1).
  \label{eq:girth-five-bound}
\end{equation}
Every graph in either class admits a signing with the corresponding bound.
\end{corollary}

The resulting coefficients of $\sqrt{d-1}$ are $\sqrt6\approx2.44949$ and $\sqrt{35/6}\approx 2.41523$. 
These bounds hold for every signing satisfying the mixed-root condition, including on non-bipartite graphs. For bipartite graphs, the theorem of Marcus,
Spielman, and Srivastava already gives the optimal existence bound~\cite{MSS15}.

\subsection*{A general spectral recovery theorem}

The key question underlying our improvement in \Cref{thm:main-signing} is the following: how much spectral information of a tuple of matrices other than the spectral edge can be recovered from the roots of the mixed determinantal polynomial?
We prove a structural result which gives an explicit second-moment refinement of the simultaneous spectral comparison in~\cite[Theorem~6]{RS21}.
Here $\diag(B)$ is the diagonal matrix with the same diagonal as $B$.

\begin{theorem}[Spectral recovery for matrix tuples]
\label{thm:tuple-recovery}
Let $A_1,\ldots,A_k$ be zero-diagonal Hermitian matrices. If the roots
of the mixed determinantal polynomial $\chi[A_1,\ldots,A_k]$ lie in $[\ell,r]$, then
\begin{equation}
  \left(A_j-\frac{k(\ell+r)}2I\right)^2
  +\sum_{i\ne j}\diag(A_i^2)
  \preceq\frac{k^2(r-\ell)^2}{4}I, \quad \forall j \in [k].
  \label{eq:tuple-interval}
\end{equation}
In particular, if the roots lie in $[-R,R]$, then
\begin{equation}
  A_j^2+\sum_{i\ne j}\diag(A_i^2)\preceq k^2R^2I.
  \label{eq:tuple-symmetric}
\end{equation}
Equality in \eqref{eq:tuple-symmetric} can hold simultaneously for all
$j$.
\end{theorem}

The proof constructs positive matrix weights supported on the mixed
roots. Their mean is $A_j/k$, and their variance is
$k^{-2}\sum_{i\ne j}\diag(A_i^2)$; see \Cref{thm:matrix-measure}.
When $k=2$ and $A_1 = A, A_2 = -A$, \eqref{eq:tuple-symmetric} gives
\begin{equation}
  A^2+\diag(A^2)\preceq4r_A^2I.
  \label{eq:intro-pair-comparison}
\end{equation}
Since $\diag(A^2)=dI$ for a signing of a $d$-regular graph, this already
implies $\rho(A)\leq\sqrt{7d-8}$ under the mixed-root condition, which improves \eqref{eq:RS}.
%By considering the fourth moment, we can strengthen \eqref{eq:intro-pair-comparison} via some graph structures, which eventually leads to \Cref{thm:main-signing} and \Cref{cor:short-cycles}.
Estimates for short walks, together with the third and fourth moments strengthen \eqref{eq:intro-pair-comparison}, yielding
\Cref{thm:main-signing} and \Cref{cor:short-cycles}.

\subsection*{Limitation of the conversion step}

We complement \Cref{thm:main-signing} by giving a lower bound on the optimal universal conversion constant
\begin{equation}
  C_{\rm mix}:=
  \sup_{\substack{d\geq2\\G\text{ finite }d\text{-regular}}}
  \ \max_{\sigma:\,r_{A_\sigma}^2\leq2(d-1)}
  \frac{\rho(A_\sigma)}{\sqrt{d-1}}.
  \label{eq:Cmix}
\end{equation}
Note that the inner set is nonempty by the interlacing argument recalled above.

\begin{theorem}[Lower bound on universal conversion]
\label{thm:lower}
The complete graph $K_8$ has a signing with signed adjacency matrix $A$ such that
\[
  r_A^2=7+2\sqrt6<12=2(d-1),\qquad \rho(A)=4+\sqrt5,
\]
where $d=7$. Consequently,
\begin{equation}
  \frac{4+\sqrt5}{\sqrt6}\leq C_{\rm mix}
  \leq\frac{3+\sqrt5}{2}.
  \label{eq:Cmix-bounds}
\end{equation}
\end{theorem}

The lower bound is approximately $2.54586$. 
It shows that the mixed-root condition allows signings whose spectral radius exceeds the Ramanujan threshold. 
Therefore, the mixed-root condition alone cannot guarantee the Ramanujan
bound for every signing satisfying it.

\subsection{Technical overview}
\label{sec:overview}

Fix a signed adjacency matrix $A$ of a $d$-regular graph satisfying the mixed-root condition. 
The roots of $\chi[A,-A]$ lie in $[-r_A,r_A]$, but they are not the eigenvalues of $A$. 
To relate these two sets of numbers, we assign probabilities to the mixed roots in each eigenvector direction. 
More precisely, for a unit eigenvector of $A$ with eigenvalue $\theta$, we construct a real random variable $X$ supported on the roots of $\chi[A,-A]$ such that
\[
  \E X=\frac\theta2,\qquad
  \E X^2=\frac{\theta^2+d}{4}.
\]
Since $|X|\leq r_A$, the first identity gives $|\theta| \leq 2 r_A$, which suffices for \eqref{eq:RS}.
The second identity is the first source of improvement, which gives $\theta^2+d\leq4r_A^2$. 
The squared eigenvalue controls both ends of the spectrum at once, yielding $\rho(A)\leq\sqrt{7d-8}$.

The main challenge is to construct an appropriate probability measure.
The starting point is to observe a random bipartite cut representation of the mixed determinantal polynomial $\chi[A,-A]$. 
Let $\eps \in \{\pm 1\}^n$ be a uniformly independent random vertex signing and $D_{\eps} = \diag(\eps_1, \ldots, \eps_n)$. 
Then, $M_{\eps} = (A - D_{\eps} A D_{\eps})/2$ is the signed adjacency matrix of a bipartite cut defined by $\eps$.
We can show that the expected characteristic polynomial $\E_{\eps} \det(xI - M_{\eps})$ is exactly $\chi[A,-A]$.
The next step is to consider a matrix-valued rational function defined as the ratio of the expected adjugate and expected characteristic polynomial
\[ F_A(x) = \frac{\E_{\eps} \adj(xI - M_{\eps})}{\E_{\eps} \det(xI - M_{\eps}) }, \]
which is an analogue of a resolvent.
The essential step is to prove that $F_A$ admits a positive-semidefinite residue decomposition
\[
	F_A(x) = \sum_{\lambda \in \Lambda_A} \frac{W_{\lambda}}{x - \lambda}, \qquad W_{\lambda} \succeq 0 \quad \text{and} \quad \sum_{\lambda \in \Lambda_A} W_{\lambda} = I,
\]
where $\Lambda_A$ is the set of distinct roots of $\chi[A,-A]$.
Thus, $\{W_{\lambda}\}_{\lambda}$ yields a matrix-valued probability measure.
The moments can be obtained by expanding $F_A(x)$ at infinity, which yields
\[
	\sum_{\lambda \in \Lambda_A} \lambda W_{\lambda} = \frac12 A, \qquad \sum_{\lambda  \in \Lambda_A} \lambda^2 W_{\lambda} = \frac14 (A^2 + \diag(A^2)).
\]
Let $u \in \R^n$ be a unit eigenvector of $A$ corresponding to $\theta$.
Then, $\{u^T W_{\lambda} u\}_{\lambda}$ defines a probability measure for $X$ with the desired moments mentioned above.

The stronger bound in \Cref{thm:main-signing} uses the same probability measure and the elementary inequality $\E X^4\leq r_A^2\E X^2$. 
Its effectiveness depends on expressing the fourth moment in terms of $A$. 
Powers of a signed adjacency matrix count walks, with each walk weighted by the product of its edge signs. 
The fourth-moment formula therefore involves short walks and signed triangles. 
We estimate the triangle terms together with the positive terms that accompany them. Completing squares over edges and restricting $A$ to a two-dimensional space at each vertex gives a lower bound on $\E X^4$ in terms of $\rho(A)$ and $d$.
Comparing it with the root-support upper bound yields the coefficient $(3+\sqrt5)/2$. 
Without triangles, several terms vanish. Without four-cycles as well, the closed walks of length four are all backtracking walks, which explains the further improvement in \Cref{cor:short-cycles}.

The key in the above plan is a matrix-valued probability measure argument with a pair of matrices $(A, -A)$ as input.
We generalize the argument for a tuple of Hermitian matrices, where a Fourier change of basis replaces the random-cut construction.
This gives the common matrix measure construction in \Cref{sec:measure}
and the exact variance behind \Cref{thm:tuple-recovery}.

\subsection{Relation to prior work and recent progress}

Our proof builds on established results on stable polynomials and mixed determinants.
Borcea and Br\"and\'en developed a stability framework for mixed principal-minor determinants~\cite{BB08}. 
Ravichandran and Srivastava applied mixed determinantal polynomials to simultaneous spectral control and graph signings~\cite{RS21}. 
The transversal polynomial used in \Cref{sec:transversals} is the $(k,n)$-characteristic polynomial of Xu, Xu, and Zhu~\cite{XXZ23}, who established its differential formula and real-rootedness in Propositions~3.3 and~3.5, respectively.

The scalar interlacing argument for positive residues and the variance bound for measures with bounded support are standard; the latter underlies the Bhatia--Davis inequality~\cite{BD00}. 
Our contribution is the construction of a positive matrix-valued measure supported on the mixed roots, together with explicit moment identities and estimates for graph signings that yield an improved two-sided spectral bound.

Since the first version of this manuscript appeared on arXiv, there have been several advances on the Bilu--Linial conjecture.
Most notably, Xu~\cite{Xu26a} refutes the conjecture by constructing a finite connected simple cubic graph with all edge signing $\sigma$ satisfying $\rho(A_\sigma)>2\sqrt{2}$.

On the algorithmic side, Jadbabaie, Saberi, and Sra~\cite{JSS26} developed a randomized recursive repair algorithm for signing a prescribed graph.
For any finite simple graph of maximum degree $d\ge 3$, their Las Vegas algorithm constructs a signing $\sigma$ satisfying $\rho(A_\sigma)<2\sqrt{2(d-1)}$ in expected polynomial time.
For bipartite graphs, the one-sided version of their algorithm produces a signing below the Ramanujan threshold $\rho(A_\sigma)<2\sqrt{d-1}$.
Subsequently, Xu~\cite{Xu26b} derandomized this framework, obtaining a deterministic polynomial-time algorithm with the same guarantee for bipartite graphs.

\subsection{Organization}

\Cref{sec:preliminaries} records notation and background on mixed polynomials, transversal compressions, and signing selection.
\Cref{sec:measure} proves the matrix measure theorem and derives its
spectral consequences. 
\Cref{sec:main-result} computes and estimates the third and fourth moments for graph signings.
\Cref{sec:signing-bounds} proves our graph signing bounds, and \Cref{sec:lower-bound} gives the $K_8$ lower bound example. 
\Cref{sec:conclusion} provides a concluding remark.

\section{Notation and polynomial background}
\label{sec:preliminaries}

\subsection{Notation}
\label{sec:notation}

In this paper, all graphs are finite, simple, and undirected. 
For $V=[n]$, a signing $\sigma:E\to\{-1,+1\}$ gives a real symmetric matrix $A_\sigma$ whose $(i,j)$-entry is the sign of $\{i,j\}$ on edges and is zero otherwise.
For a Hermitian matrix $A$, its spectral radius is $\rho(A)=\norm A_{\rm op}$. 
We write $A^*$ for the adjoint, $A\succeq0$ for positive semidefiniteness, and $A\preceq B$ when $B-A\succeq0$. 
The identity matrix is denoted by $I$, with its size understood from the context. The entrywise product is denoted by $\circ$.

We use $\diag(A_1, \ldots, A_k)$ to denote block-diagonal matrix formed from square matrices $A_1, \ldots, A_k$.
For the ease of notation, we also use $\diag(A)$ to denote the diagonal matrix with the same diagonal as $A$.
For Hermitian $A$,
\[
  (\diag(A^2))_{ii}=\sum_j|a_{ij}|^2,
  \qquad \Delta_A:=\diag(A^2)=dI
  \quad\text{for a signed $d$-regular graph}.
\]
For $S\subseteq[n]$, $A[S]$ denotes the corresponding principal submatrix. 
The characteristic polynomial of the empty submatrix is $1$.
For a nonconstant real-rooted polynomial $p$, $\maxroot p$ denotes its largest root.

\subsection{Mixed determinantal polynomials}
\label{sec:mixed-background}

For Hermitian $n\times n$ matrices $A_1,\ldots,A_k$, define the normalized mixed determinantal polynomial by
\begin{equation}
  \chi[A_1,\ldots,A_k](x)
  :=k^{-n}\sum_{S_1\sqcup\cdots\sqcup S_k=[n]}
               \prod_{j=1}^k\det(xI-A_j[S_j]).
  \label{eq:mixed-determinantal-polynomial}
\end{equation}
The sum is over ordered partitions, allowing empty parts. 
The polynomial is monic of degree $n$ and is invariant under permuting its input matrices. 
It is also real-rooted and the transversal formula below provides a determinantal proof.

For a real symmetric zero-diagonal matrix $A$, we use
\begin{equation*}
  q_A(x):=\chi[A,-A](x),\qquad r_A:=\maxroot q_A.
\end{equation*}
Exchanging the two parts in \eqref{eq:mixed-determinantal-polynomial} gives $q_A(-x)=(-1)^nq_A(x)$. 
Its roots therefore belong to $[-r_A,r_A]$, and $q_{-A}=q_A$.

\subsection{Transversal compressions}
\label{sec:transversals}

Partition the coordinates of $\C^{kn}$ into the groups $\{i,n+i,\ldots,(k-1)n+i\}$, for $i\in[n]$. 
A \emph{transversal} chooses one coordinate from each group. When viewing the compression as an $n\times n$ matrix, we order its coordinates by $i$.
For a Hermitian $kn\times kn$ matrix $H$, define
\[
  \psi_H(x):=k^{-n}\sum_{T\text{ transversal}}\det(xI-H[T]).
\]
This is the $(k,n)$-characteristic polynomial of Xu, Xu, and Zhu~\cite{XXZ23}.
We record their differential formula and real-rootedness result, together with the unitary invariance that the formula gives.
Here we include a proof sketch for the sake of completeness.

\begin{lemma}[Proposition 3.3 and Proposition 3.5 of \cite{XXZ23}]
\label{lem:transversal-compressions}
Let $Z=\diag(z_1,\ldots,z_n)$ and $Z_k=I_k\otimes Z$. Then
\begin{equation*}
  \psi_H(x)=\left.
  \frac1{(k!)^n}\prod_{i=1}^n\partial_{z_i}^{\,k-1}\det(Z_k-H)
  \right|_{z_1=\cdots=z_n=x}.
\end{equation*}
The polynomial $\psi_H$ is real-rooted. Moreover,
$\psi_{U^*HU}=\psi_H$ for every unitary matrix $U$ commuting with
$Z_k$ for all choices of $z_1,\ldots,z_n$.
\end{lemma}

\begin{proof}
The differential formula and real-rootedness are \cite[Propositions~3.3 and~3.5]{XXZ23}. 
To recall the normalization, differentiating $k-1$ times in $z_i$ deletes all but one coordinate of the $i$th group, with factor $(k-1)!$. 
The total factor is $((k-1)!)^n$, which becomes $k^{-n}$ after division by $(k!)^n$.
Finally, commutation with $Z_k$ gives $\det(Z_k-U^*HU)=\det(Z_k-H)$, so the differential formula gives the stated invariance.
\end{proof}

Taking $H=\diag(A_1,\ldots,A_k)$ identifies $\psi_H$ with $\chi[A_1,\ldots,A_k]$ (see \cite[eq. (3.6)]{XXZ23}), i.e.,
\begin{equation} \label{eq:mixed-psi}
	\chi[A_1, \ldots, A_k](x) = \psi_{\diag(A_1,\ldots,A_k)}(x)
\end{equation}
is a real-rooted polynomial.
These facts hold for arbitrary Hermitian inputs, including after a rank-one perturbation, e.g., $\chi[A_1 + tvv^*, \ldots, A_k]$ is also real-rooted for any $t \in \R$.

\subsection{Positive residues of a real-rooted pencil}

We will also need the following simple corollary of the Hermite–Kakeya–Obreschkoff theorem and the partial-fraction characterization of interlacing.
See, e.g., \cite[Section~2.3, especially Exercise~2.5]{Wag11}.
We include a proof sketch for completeness.

\begin{lemma}[Positive residues of a real-rooted pencil] \label{lem:residue-decomposition} 
Let $q$ be monic of degree $n$, and let $g$ have degree $n-1$ and positive leading coefficient. 
If $q-tg$ is real-rooted for every $t\in\R$, then $g$ is real-rooted, $g/q$ has only simple poles, and 
\[ \frac{g(x)}{q(x)}=\sum_{\lambda:q(\lambda)=0} \frac{c_\lambda}{x-\lambda}, \qquad c_\lambda\geq0, \] 
where the sum is over distinct roots and $c_\lambda=0$ whenever $g/q$ is analytic at $\lambda$. 
\end{lemma}

\begin{proof}
The polynomial $g$ is real-rooted, since it is the nonzero coefficientwise limit of the real-rooted polynomials $g-q/t$ as $t\to+\infty$. 
Hence every nonzero real linear combination of $q$ and $g$ is real-rooted. 
By the Hermite--Kakeya--Obreschkoff theorem \cite[Theorem~6.3.8]{RS02}, $g$ interlaces $q$.
If $\lambda$ is a root of $q$ of multiplicity $m$, the interlacing inequalities force $\lambda$ to be a root of $g$ of multiplicity at least $m-1$. 
Therefore, every pole of $g/q$ is simple.

Let $h$ be their monic greatest common divisor, and write $q=hQ$ and $g=hG$. 
Then, $Q$ has simple real roots, strictly interlaced by those of $G$. 
Since $Q$ is monic, $G$ has positive leading coefficient, and $\deg G=\deg Q-1$, partial fractions give
\[
  \frac{g(x)}{q(x)}
  =\frac{G(x)}{Q(x)}
  =\sum_{\lambda:Q(\lambda)=0}
      \frac{G(\lambda)}{Q'(\lambda)}\,\frac1{x-\lambda}.
\]
Strict interlacing implies that $G(\lambda)$ and $Q'(\lambda)$ have the same sign at every root of $Q$, so each displayed residue is positive. 
Assigning residue zero at roots of $q$ where the quotient is analytic proves the claim.
\end{proof}

\subsection{Signings with small mixed roots}
\label{sec:signing-input}

The graph signing bounds use the following selection result from~\cite{RS21}, which crucially uses the interlacing family technique.
%Again, we include a proof sketch for completeness.

\begin{lemma}[Ravichandran--Srivastava mixed-root bound~\cite{RS21}]
\label{lem:mixed-root-signing}
Let $G$ have maximum degree at most $d$, where $d\geq2$. 
There exists a signing $\sigma$ such that $r_{A_\sigma}^2\leq2(d-1)$.
\end{lemma}

\section{Positive matrix measures and spectral recovery}
\label{sec:measure}

In this section, we prove the existence of the positive matrix-valued measure on the mixed roots for each input matrix as mentioned in \Cref{sec:overview}.
The following is the main theorem in this section, which constructs the measure and identifies its first two moments and variance.

\begin{theorem}[Positive matrix-valued measure and exact variance]
\label{thm:matrix-measure}
Let $A_1,\ldots,A_k$ be zero-diagonal Hermitian $n\times n$ matrices, and let $\Lambda$ be the set of distinct roots of $q=\chi[A_1,\ldots,A_k]$. 
For each $j\in[k]$, there are Hermitian
matrices $\{W_\lambda^{(j)}\succeq0$, $\lambda\in\Lambda\}$, satisfying
\begin{align}
  \sum_\lambda W_\lambda^{(j)}=I, \quad
  \sum_\lambda\lambda W_\lambda^{(j)}=\frac{A_j}{k}, \quad \text{and}~~
  \sum_\lambda\lambda^2W_\lambda^{(j)}=\frac{A_j^2+\sum_{i\ne j}\diag(A_i^2)}{k^2}.
  \label{eq:tuple-moments}
\end{align}
Furthermore, for $H_m^{(j)} :=\sum_\lambda\lambda^mW_\lambda^{(j)}$, it holds that
\begin{equation}
  H_2^{(j)}-(H_1^{(j)})^2
  =\frac1{k^2}\sum_{i\ne j}\diag(A_i^2).
  \label{eq:tuple-variance}
\end{equation}
For real symmetric inputs, the matrices $W_\lambda^{(j)}$ may be chosen real symmetric.
\end{theorem}

Before proving the theorem, we first introduce a representation of the mixed determinantal polynomials that is preserved under rank-one perturbation.
The representation is based on the following Fourier coupling.
Let $\eta_1,\ldots,\eta_n$ be independent and uniform on $\{0,\ldots,k-1\}$, let $\omega=\exp(2\pi\mathrm{i}/k)$ and define
\begin{equation}
  D_\eta=\diag(\omega^{\eta_1},\ldots,\omega^{\eta_n}),\qquad
  M_\eta=\frac1k\sum_{a=0}^{k-1}D_\eta^a A_{a+1}D_\eta^{-a}.
  \label{eq:phase-matrix}
\end{equation}
Note that all summands use powers of the same diagonal phase matrix.

\begin{lemma}[Rank-one perturbation preserving representation] \label{lem:pencil}
For any Hermitian tuple $(A_1, \ldots, A_k)$, for any $v\in\C^n$ and $t\in\R$, we have
\begin{align}
  q(x)= \chi[A_1, \ldots, A_k](x) &=\E_\eta\det(xI-M_\eta),
  \label{eq:phase-characteristic}\\
  \chi[A_1+ktvv^*,A_2,\ldots,A_k](x)
  &=\E_\eta\det(xI-M_\eta-tvv^*).
  \label{eq:phase-rank-one}
\end{align}
In particular, \eqref{eq:phase-rank-one} is real-rooted for all $t \in \R$.
\end{lemma}

\begin{proof}
To prove these identities, index the coordinates of $\C^{kn}$ by $(a,i)$, where $0\le a<k$ and $i\in[n]$. Let $U$ be the block Fourier matrix with $n\times n$ blocks $U_{ab}=k^{-1/2}\omega^{-ab}I$, for $0\leq a,b<k$.
Let $H_t=\diag(A_1+ktvv^*,A_2,\ldots,A_k)$, block multiplication gives that
\[
  (U^*H_tU)_{bc}
  =\frac1k\sum_{a=0}^{k-1}\omega^{a(b-c)}A_{a+1}+tvv^*.
\]
Given $\eta$, then $T = \{(\eta_i, i)\}_{i \in [n]}$ forms a transversal, and $T$ indexed principal submatrix is
\[
(U^* H_t U)[T] = \left( \frac{1}{k} \sum_{a=0}^{k-1} \omega^{a(\eta_i - \eta_j)} (A_{a+1})_{ij} + t v_i \overline{v_j} \right)_{i,j=1}^n = M_{\eta} + t v v^*.
\]
Thus, according to the definition of $\psi_H$, it holds that
\begin{align*}
\psi_{U^* H_t U}(x) & = \frac{1}{k^n} \sum_{T\text{ transversal}}\det(xI- (U^*H_tU)[T]) \\
& = \frac{1}{k^n} \sum_{\eta} \det(xI - M_{\eta} - t v v^* ) = \E_{\eta} \det(xI - M_{\eta} - t v v^*).
\end{align*}
Since $U$ is unitary and commutes with $I_k \otimes \diag(z_1, \ldots, z_n)$, \Cref{lem:transversal-compressions} implies that $\psi_{H_t} = \psi_{U^* H_t U}$.
Combining with \eqref{eq:mixed-psi} proves \eqref{eq:phase-rank-one}, and taking $t=0$ gives \eqref{eq:phase-characteristic}.
Real-rootedness also follows from \Cref{lem:transversal-compressions}.
\end{proof}

Now, we are ready to prove the main result in this section.

\begin{proof}[Proof of \Cref{thm:matrix-measure}]
It suffices to prove the theorem for $j=1$, as all other choices follow by reordering the tuple. 

In order to construct the positive matrix-valued measure, a key step is to introduce the following matrix-valued rational function.
Define
\begin{equation}
  F(x):=\frac{G(x)}{q(x)}, \qquad \text{where} \quad G(x):=\E_\eta\adj(xI-M_\eta).
  \label{eq:tuple-quotient}
\end{equation}

For any nonzero $v \in \C^n$, we consider the polynomial $g_v(x)=v^*G(x)v$, which has degree $n-1$ and leading coefficient $\norm v_2^2 > 0$. Since
\[
  q(x)-tg_v(x) = \E_{\eta}[\det(xI - M_{\eta}) - t v^* \adj(x I - M_{\eta})v],
\]
the matrix determinant lemma and \eqref{eq:phase-rank-one} imply that
\[
q(x)-tg_v(x) = \E_{\eta} \det(xI - M_{\eta} - t vv^*) = \chi[A_1+ktvv^*,A_2,\ldots,A_k](x).
\]
\Cref{lem:pencil} shows that $q(x)-tg_v(x)$ is real-rooted for all $t \in \R$.
Thus, it follows from \Cref{lem:residue-decomposition} that $g_v$ is real-rooted and $g_v/q$ has only simple poles, and admits a nonnegative residue decomposition
\[ 
\frac{g_v(x)}{q(x)}=\sum_{\lambda \in \Lambda} \frac{c_{\lambda,v}}{x-\lambda}, \qquad c_{\lambda,v}\geq0, \] 
where $c_{\lambda,v}=0$ whenever $g_v/q$ is analytic at $\lambda$.

Since $q$ has real coefficients and $G$ has Hermitian coefficients, $F(\overline z)^*=F(z)$ whenever $q(z) \neq 0$. 
Thus, the Laurent coefficients of $F$ at each $\lambda\in\Lambda$ are Hermitian. 
For every $s\geq2$, let $B_{\lambda,s}$ be the coefficient of $(x-\lambda)^{-s}$. The scalar conclusion gives $v^*B_{\lambda,s}v=0$ for every $v$, and hence $B_{\lambda,s}=0$.
Thus, $F$ itself has only simple poles. Define
\[
  W_\lambda^{(1)}:=\lim_{x\to\lambda}(x-\lambda)F(x),
\]
assigning the zero matrix when the singularity is removable.
The nonnegative scalar residues give $v^*W_\lambda^{(1)}v\geq0$ for every $v$, so $W_\lambda^{(1)}\succeq0$.
Since $\deg G<\deg q$, there is no polynomial part, and therefore
\begin{equation}
  F(x)=\sum_{\lambda\in\Lambda}
       \frac{W_\lambda^{(1)}}{x-\lambda},
  \qquad W_\lambda^{(1)}\succeq0.
  \label{eq:tuple-residues}
\end{equation}
If all input matrices are real, the substitution $\eta\mapsto-\eta$ conjugates $M_\eta$ and preserves its distribution. Hence $G$ and all residue matrices are real symmetric.

Now, expanding \eqref{eq:tuple-residues} as $x \to +\infty$ yields
\begin{align} \label{eq:residue-expansion}
  F(x)
  =\frac{\sum_{\lambda\in\Lambda} W_\lambda^{(1)}}{x}
   +\frac{\sum_{\lambda\in\Lambda}\lambda W_\lambda^{(1)}}{x^2}
   +\frac{\sum_{\lambda\in\Lambda}\lambda^2W_\lambda^{(1)}}{x^3}
   +O(x^{-4}).
\end{align}

To establish \eqref{eq:tuple-moments}, we expand $F(x)$ as $x \to +\infty$ again using the Fourier coupling representation in \eqref{eq:phase-matrix}.
Since each $M_\eta$ has zero diagonal and hence zero trace, it holds that
\begin{align*}
  \det(xI-M_\eta)
  &=x^n\left(1-\frac{\tr M_\eta^2}{2x^2}+O(x^{-3})\right),~~ \text{and} \\
  \adj(xI-M_\eta) & = \det(xI - M_{\eta}) (xI - M_{\eta})^{-1} \\
  &=x^{n-1}\left(I+\frac{M_\eta}{x}
    +\frac{M_\eta^2-\tfrac12(\tr M_\eta^2)I}{x^2}
    +O(x^{-3})\right).
\end{align*}
Taking expectation over $\eta$ gives
\begin{align*}
	q(x)^{-1} & = (\E_{\eta} \det(xI - M_{\eta}))^{-1} = x^{-n} \left(1 + \frac{\E \tr M_\eta^2}{2x^2} + O(x^{-3})\right), ~~\text{and} \\
	G(x) & = \E_{\eta} \adj(xI - M_{\eta}) = x^{n-1}\left(I+\frac{\E M_\eta}{x}
    +\frac{\E M_\eta^2- \tfrac12(\E \tr M_\eta^2)I}{x^2}
    +O(x^{-3})\right).
\end{align*}
Therefore, we obtain another expansion of $F(x)$ as $x \to +\infty$ that
\begin{equation}
  F(x) = q(x)^{-1} G(x) =\frac Ix+\frac{\E M_\eta}{x^2}
                 +\frac{\E M_\eta^2}{x^3}+O(x^{-4}).
  \label{eq:tuple-series}
\end{equation}
By comparing the coefficients in \eqref{eq:residue-expansion} and \eqref{eq:tuple-series}, we obtain that
\begin{equation*}
\sum_{\lambda\in\Lambda} W_\lambda^{(1)} = I, \quad \sum_{\lambda\in\Lambda} \lambda W_\lambda^{(1)} = \E M_{\eta}, \quad \text{and}~~ \sum_{\lambda\in\Lambda} \lambda^2 W_{\lambda}^{(1)} = \E M_{\eta}^2.
\end{equation*}

It remains to conduct some routine calculations on $\E M_{\eta}$ and $\E M_{\eta}^2$ to establish \eqref{eq:tuple-moments}.
According to the definition of $M_{\eta}$, for $i \neq j$, we have
\[
  (M_\eta)_{ij}
  =\frac1k\sum_{a=0}^{k-1}
       \omega^{a(\eta_i-\eta_j)}(A_{a+1})_{ij}.
\]
Independence and uniformity of $\eta$ give
$\E\omega^{a(\eta_i-\eta_j)} =\bigl(\E\omega^{a\eta_i}\bigr) \bigl(\E\omega^{-a\eta_j}\bigr)$, which is $1$ when $a=0$ and $0$ otherwise.
Thus, only the term $a=0$ survives on the off-diagonal entries.
Since all input matrices have zero diagonal,
\[
  \E M_\eta=\frac{A_1}{k}.
\]

For the second moment, expand the matrix product entrywise:
\[
  (M_\eta^2)_{ij}
  =\frac1{k^2}\sum_{h=1}^n\sum_{a,b=0}^{k-1}
     \omega^{a\eta_i+(b-a)\eta_h-b\eta_j}
     (A_{a+1})_{ih}(A_{b+1})_{hj}.
\]
First consider the case of $i\ne j$. 
A term with $h=i$ or $h=j$ vanishes as input matrices have zero diagonal.
For every remaining term, the three indices $i,h,j$ are distinct.
The coefficient expectation factors as
\[
  \E\omega^{a\eta_i+(b-a)\eta_h-b\eta_j}
  =\bigl(\E\omega^{a\eta_i}\bigr)
   \bigl(\E\omega^{(b-a)\eta_h}\bigr)
   \bigl(\E\omega^{-b\eta_j}\bigr).
\]
This product is nonzero exactly when $a=b=0$. Thus
\[
  (\E M_\eta^2)_{ij}
  =\frac1{k^2}\sum_{h=1}^n(A_1)_{ih}(A_1)_{hj}
  =\frac{(A_1^2)_{ij}}{k^2},\qquad i\ne j.
\]
For a diagonal entry with $i=j$, again the term $h=i$ vanishes.
For $h\ne i$, the coefficient reduces to $\omega^{(a-b)(\eta_i-\eta_h)}$. 
Its expectation is $1$ when $a=b$ and $0$ otherwise. Consequently,
\begin{align*}
  (\E M_\eta^2)_{ii} =\frac1{k^2}\sum_{a=0}^{k-1}\sum_{h=1}^n (A_{a+1})_{ih}(A_{a+1})_{hi} =\frac1{k^2}\sum_{a=1}^k(A_a^2)_{ii}.
\end{align*}
Combining the diagonal and off-diagonal formulas yields
\[
  \E M_\eta^2
  =\frac{A_1^2+\sum_{a=2}^k\diag(A_a^2)}{k^2}.
\]
This proves \eqref{eq:tuple-moments} for $j=1$.
\eqref{eq:tuple-variance} follows directly from \eqref{eq:tuple-moments}.
\end{proof}

%The trace identity shows that every distinct mixed root has a nonzero matrix weight. A directional weight $v^*W_\lambda^{(j)}v$ may still vanish. 
%The representation therefore includes multiple roots without requiring positive mass in every direction.

\subsection{Spectral bounds from the moments}

The root interval and the first two moments now give \Cref{thm:tuple-recovery} directly.

\begin{proof}[Proof of \Cref{thm:tuple-recovery}]
Let $q(x) = \chi[A_1, \ldots, A_k]$. By \Cref{thm:matrix-measure}, it holds that
\[
  0\preceq\sum_{\lambda: q(\lambda)=0}(\lambda-\ell)(r-\lambda)W_\lambda^{(j)}
  =-H_2^{(j)}+(\ell+r)H_1^{(j)}-\ell rI.
\]
Replacing the first and second moments according to \eqref{eq:tuple-moments}, then rearranging the terms gives \eqref{eq:tuple-interval}, whose symmetric specialization is \eqref{eq:tuple-symmetric}.

For sharpness, take real numbers $a_1,\ldots,a_k$, and let
\[
  A_i=a_i\begin{pmatrix}0&1\\1&0\end{pmatrix}, \quad i = 1, \ldots, k.
\]
Then, it follows from \eqref{eq:phase-characteristic} that
\begin{align*}
q(x) & = \E_{\eta} \det\left(xI - \frac{1}{k} \sum_{i=0}^{k-1} a_{i+1} \begin{pmatrix} & \omega^{i(\eta_1-\eta_2)} \\ \omega^{i(\eta_2-\eta_1)} & \end{pmatrix} \right) \\
& = \E_{\eta} \left( x^2 - \frac{1}{k^2} \sum_{i,j=0}^{k-1} a_{i+1}a_{j+1} \omega^{(i-j)(\eta_1-\eta_2)} \right) \\
& = x^2 - \frac{1}{k^2} \sum_{i,j=0}^{k-1} a_{i+1} a_{j+1} \E \omega^{(i-j)(\eta_1 - \eta_2)}.
\end{align*}
Since $\eta$ is independently and uniformly sampled from $[k]^2$, the expectation term is nonzero only when $i=j$. Therefore,
\[
  q(x)=x^2-\frac{1}{k^2} \sum_{i=1}^k a_i^2, \qquad  R^2=\frac{1}{k^2} \sum_{i=1}^k a_i^2.
\]
For every $j$, the left side of \eqref{eq:tuple-symmetric} is $(\sum_i a_i^2)I=k^2R^2I$, and the equality holds.
\end{proof}

The same measure also gives higher-moment inequalities. 
We record them to make the relation between the second- and fourth-moment arguments more explicit.

\begin{proposition}[Moment inequalities]
\label{prop:moment-transfer}
Let $W_\lambda\succeq0$ with $\sum_\lambda W_\lambda=I$, where $\lambda$ is supported on $[-R,R]$, and put $H_m=\sum_\lambda\lambda^mW_\lambda$.
If a real polynomial $p$ is nonnegative on the support, then $\sum_\lambda p(\lambda)W_\lambda\succeq0$. Moreover, for $m\geq1$, the block matrices satisfy
\begin{equation*}
  [H_{a+b}]_{a,b=0}^{m}\succeq0,\qquad
  [R^2H_{a+b}-H_{a+b+2}]_{a,b=0}^{m-1}\succeq0.
\end{equation*}
In particular, $H_{2s}\preceq R^2H_{2s-2}$ for every $s\geq1$.
\end{proposition}

\begin{proof}
The polynomial assertion is a sum of positive semidefinite matrices.
For a vector $z = (z_0,\ldots,z_m) \in \C^{(m+1)n}$, where $z_0, \ldots, z_m \in \C^n$, the quadratic form of $[H_{a+b}]_{a,b=0}^{m}$ with respect to $z$ is
\[
  z^* \Big( [H_{a+b}]_{a,b=0}^{m} \Big) z = \sum_\lambda\left(\sum_{a=0}^m\lambda^az_a\right)^* W_\lambda\left(\sum_{a=0}^m\lambda^a z_a\right)\geq 0,
\]
since $W_\lambda \succeq 0$.
Similarly, the quadratic form of $[R^2H_{a+b}-H_{a+b+2}]_{a,b=0}^{m-1}$ with respect to $z' = (z_0, \ldots, z_{m-1}) \in \C^{mn}$ is
\[
(z')^* \Big( [R^2H_{a+b}-H_{a+b+2}]_{a,b=0}^{m-1} \Big) z' = \sum_{\lambda} (R^2 - \lambda^2) \left(\sum_{a=0}^{m-1} \lambda^az_a\right)^* W_\lambda \left(\sum_{a=0}^{m-1} \lambda^az_a\right) \geq 0.
\]
The last assertion follows by taking $p(\lambda)=\lambda^{2s-2}(R^2-\lambda^2)$ in $\sum_\lambda p(\lambda)W_\lambda\succeq0$.
\end{proof}

\subsection{The pair \texorpdfstring{$(A,-A)$}{(A,-A)}}

For the pair $(A,-A)$, the Fourier average in \eqref{eq:phase-matrix} becomes the random-cut matrix
\begin{equation*}
  M_\varepsilon:=\frac{A-D_\varepsilon AD_\varepsilon}{2}, \quad \text{where}~~D_\varepsilon:=\diag(\varepsilon_1,\ldots,\varepsilon_n)
%  \label{eq:cut-matrix}
\end{equation*}
and the random vertex signs $\eps_1, \ldots, \eps_n$ are independent and uniform over $\{-1,+1\}$.
Thus, $(M_\varepsilon)_{ij}=a_{ij} \mathbf1_{\{\varepsilon_i\ne\varepsilon_j\}}$, and $M_\varepsilon$ is bipartite. 
The construction gives
\begin{equation}
  q_A(x)=\E_\varepsilon\det(xI-M_\varepsilon).
  \label{eq:random-cut-characteristic}
\end{equation}
The matrix-valued rational function in \eqref{eq:tuple-quotient} becomes
\begin{equation}
  F_A(x)=\frac{G_A(x)}{q_A(x)}
         =\sum_{\lambda:q_A(\lambda)=0}\frac{W_\lambda}{x-\lambda}, \qquad \text{where} \quad G_A(x)=\E_\varepsilon\adj(xI-M_\varepsilon).
  \label{eq:averaged-resolvent}
\end{equation}
Expanding $F_A(x)$ as $x \to +\infty$ gives
\begin{equation*}
	F_A(x)=\sum_{m\geq0}\frac{\Hmom_m}{x^{m+1}}\quad(|x|>r_A), \quad \text{where}~~ \Hmom_m=\sum_{\lambda: q_A(\lambda)=0} \lambda^mW_\lambda
%  \label{eq:matrix-moments}
\end{equation*}

The following corollary is a simple consequence of applying \Cref{thm:matrix-measure}, \Cref{thm:tuple-recovery}, and \Cref{prop:moment-transfer} to $(A,-A)$.
\begin{corollary}[Recovery for a signed pair]
\label{cor:pair-recovery}
For every real symmetric zero-diagonal matrix $A$, the matrix coefficients in \eqref{eq:averaged-resolvent} are positive semidefinite and supported on $[-r_A,r_A]$. 
Their first moments are
\begin{equation*}
  \Hmom_0=I,\qquad \Hmom_1=\frac A2,\qquad
  \Hmom_2=\frac{A^2+\Delta_A}{4}.
\end{equation*}
Consequently,
\begin{equation}
  A^2+\Delta_A\preceq4r_A^2I,\qquad
  \Hmom_{2s}\preceq r_A^2\Hmom_{2s-2}\quad(s\geq1).
  \label{eq:matrix-resolvent}
\end{equation}
\end{corollary}

\section{Third and fourth moments of graph signings}
\label{sec:main-result}

Fix a signed adjacency matrix $A$ of a $d$-regular graph and set $\rho:=\norm A_{\rm op}$. 
Note that $\sqrt d\leq\rho\leq d$.
Replacing $A$ by $-A$ preserves $q_A$, so we may choose a unit eigenvector $u$ with $Au=\rho u$ and $\|u\|_2 = 1$.

The matrix weights from \Cref{cor:pair-recovery} give a scalar probability measure and its moments
\[
  \mu_u:=\sum_{\lambda:q_A(\lambda)=0}(u^\top W_\lambda u)\,\delta_\lambda,
  \qquad m_s:=\int\lambda^s\,d\mu_u(\lambda)=u^\top\Hmom_su.
\]
The support is contained in $[-r_A,r_A]$, and
\begin{equation*}
  m_0=1,\qquad m_1=\frac\rho2,\qquad
  m_2=\frac{\rho^2+d}{4}.
\end{equation*}
Furthermore, if $A$ satisfies the mixed-root condition, then the moment relationship in \eqref{eq:matrix-resolvent} yields
\begin{equation}
  m_4\leq r_A^2m_2
       \leq\frac12(\rho^2+d)(d-1).
  \label{eq:fourth-moment-upper}
\end{equation}
We will compare this upper bound with a lower bound obtained from the short walks of the graph. The moment identities and lower estimates below hold for every signing of a $d$-regular graph, without the mixed-root assumption.

\subsection{The third-moment and fourth-moment identities}

We first compute the matrix-valued moments $\Hmom_3$ and $\Hmom_4$.

\begin{lemma}
\label{lem:fourth-moment-matrix}
Let $t_2=\tr(M_\varepsilon^2)$. Then
\begin{equation}
  \Hmom_j
  =
  \E M_\varepsilon^j
  -\frac12\Cov(t_2,M_\varepsilon^{j-2}),
  \qquad j=3,4,
  \label{eq:fourth-moment-covariance}
\end{equation}
where
\[
  \Cov(t_2,M_\varepsilon^{j-2})
  =
  \E(t_2M_\varepsilon^{j-2})
  -(\E t_2)(\E M_\varepsilon^{j-2}).
\]
In particular,
\begin{equation} \label{eq:exact-third-matrix}
  8\Hmom_3 = A^3+(2d-1)A-\diag(A^3),
\end{equation}
and
\begin{equation}
  16\Hmom_4 ={}A^4+(3d-2)A^2+d^2I+(A\circ A)\circ A^2 -A\diag(A^3)-\diag(A^3)A+\diag(A^4).
\label{eq:exact-fourth-matrix}
\end{equation}
\end{lemma}

The proof, based on coefficient comparison in the power-series
expansion of \eqref{eq:averaged-resolvent} and averaging over
the random vertex signs, is given in \Cref{app:fourth-moment}.

\begin{lemma}[Fourth moment in an extremal direction]
\label{lem:exact-fourth-moment}
Let $Au=\rho u$ with $\norm u_2=1$, and put
\[
  \tau:=\sum_{i=1}^n u_i^2(A^3)_{ii}.
\]
Then
\begin{align}
  16m_4
  &=\rho^4+(3d-2)\rho^2+d^2 -2\rho\tau +\sum_{i=1}^n u_i^2(A^4)_{ii} + 2\sum_{\{i,j\}\in E}u_iu_j(A^2)_{ij}.
  \label{eq:exact-fourth-scalar}
\end{align}
\end{lemma}

\begin{proof}
Since $u$ is an eigenvector of $A$ corresponding to the eigenvalue $\rho$, it holds that
\[
u^\top A\diag(A^3) u = u^\top \diag(A^3) A u = \rho \tau.
\]
Moreover, $(A\circ A)\circ A^2$ is the entrywise restriction of $A^2$ to the edges of $G$.
Therefore, the lemma follows by taking the quadratic form at $u$ on both sides of \eqref{eq:exact-fourth-matrix} in \Cref{lem:fourth-moment-matrix}. 
\end{proof}

\subsection{Estimating the triangle term}

The term $\tau$ is related to triangles in the input graph, and we call it triangle term.
It turns out that $\tau$ is very useful in deriving a lower bound on the fourth-moment.
We first lower bound the last two terms of \eqref{eq:exact-fourth-scalar} in terms of $d$, $\rho$ and $\tau$.

\begin{lemma}
\label{lem:edgewise-decomposition}
With the notation above,
\begin{align}
  \sum_{i=1}^n u_i^2(A^4)_{ii} + 2\sum_{\{i,j\}\in E}u_iu_j(A^2)_{ij} \geq d^2 +\frac{(\tau+\rho/2)^2}{d} - \frac d4.
  \label{eq:reduced-fourth-moment-bound}
\end{align}
\end{lemma}

\begin{proof}
We start with the first term. Since $(A^2)_{ii} = d$ and $\|u\|_2=1$, we have
\begin{equation*}
  \sum_{i=1}^n u_i^2(A^4)_{ii} = d^2+ \sum_{i=1}^n u_i^2 \sum_{j\ne i}(A^2)_{ij}^2 \geq d^2+ \sum_{\{i,j\} \in E} (u_i^2 + u_j^2) (A^2)_{ij}^2,
\end{equation*}
where the inequality holds by omitting those nonnegative terms on non-edges.
Then, the remaining goal is to lower bound
\begin{equation} \label{eq:two-sum}
\sum_{\{i,j\} \in E} \Big( (u_i^2 + u_j^2) (A^2)_{ij}^2 + 2u_iu_j(A^2)_{ij} \Big).
\end{equation}
For an edge $e=\{i,j\}$, we write
\[
  w_e=u_i^2+u_j^2,\qquad x_e=a_{ij}(A^2)_{ij},\qquad
  b_e=a_{ij}u_iu_j.
\]
Then, \eqref{eq:two-sum} becomes $\sum_{e \in E} (w_ex_e^2+2b_ex_e)$.
Note that, if $w_e = 0$ then $u_i=u_j= 0$ and $b_e=0$, so this edge contributes zero to the sum.
Thus, we may assume the summation is over edges with $w_e > 0$ and write
\begin{equation} \label{eq:edge-sums}
\sum_{e \in E} (w_ex_e^2+2b_ex_e) = \sum_{e \in E} w_e\left(x_e+\frac{b_e}{w_e}\right)^2  -\sum_{e \in E} \frac{b_e^2}{w_e}.
\end{equation}
Before we proceed, it is helpful to observe that
\[
\sum_{e \in E}  w_e x_e = \sum_{i=1}^n u_i^2 \sum_{j\neq i} a_{ij} (A^2)_{ij} = \sum_{i=1}^n u_i^2 (A^3)_{ii} = \tau \quad \text{and} \quad \sum_{e \in E} b_e = \frac12 u^\top A u = \frac{\rho}{2}.
\]
Similarly, as the input is $d$-regular and $\|u\|_2=1$, we have
\[
\sum_{e \in E} w_e = \sum_{i=1}^n u_i^2 \sum_{j\neq i} 1_{\{i,j\} \in E} = d.
\]
Now, using $b_e^2 = u_i^2 u_j^2 \leq w_e^2/4$ and $\sum_e w_e = d$ to control the negative part of \eqref{eq:edge-sums}
\[ \sum_{e \in E} \frac{b_e^2}{w_e} \leq \frac{d}{4}. \]
Finally, Cauchy-Schwarz together with $\sum_e w_e = d$ yield that
\[
(\tau + \rho/2)^2 = \left( \sum_{e \in E} (w_e x_e + b_e) \right)^2 \leq d \sum_{e \in E} w_e\left(x_e+\frac{b_e}{w_e}\right)^2.
\]
Substituting these bounds into \eqref{eq:edge-sums} proves the lemma.
\end{proof}

%\rev{
%The difference between the two sides of \eqref{eq:reduced-fourth-moment-bound} has an explicit nonnegative form. 
%Retain the notation $w_e,x_e,b_e$ from \Cref{lem:edgewise-decomposition}, and put
%\[
%  y_e=x_e+b_e/w_e\quad(w_e>0),\qquad
%  \bar y=(\tau+\rho/2)/d,
%\]
%with $y_e=0$ for $w_e=0$. Let
%\[
%  R_0=\sum_i u_i^2
%  \sum_{\substack{j\ne i\\\{i,j\}\notin E}}(A^2)_{ij}^2.
%\]
%The left side minus the right side of \eqref{eq:reduced-fourth-moment-bound} is exactly
%\begin{equation*}
%  R_0+\sum_e w_e(y_e-\bar y)^2
%  +\frac14\sum_{\substack{e=\{i,j\}\\w_e>0}}
%                 \frac{(u_i^2-u_j^2)^2}{w_e}.
%\end{equation*}
%The terms are the discarded nonedge squares, the weighted variance in Cauchy--Schwarz, and the remainder in $u_i^2u_j^2\leq(u_i^2+u_j^2)^2/4$.}

It remains to bound $\tau$. Although signed triangle counts need not
be positive, the largest eigenvalue bounds them at each vertex.

\begin{lemma}[A local triangle bound]
\label{lem:diagonal-third-moment}
For every vertex $i$,
\begin{equation}
  (A^3)_{ii}\leq d\rho-\frac{d^2}{\rho}.
  \label{eq:triangle-compression}
\end{equation}
Consequently, $\tau\leq d\rho-d^2/\rho$.
\end{lemma}

\begin{proof}
The vectors $e_i$ and $Ae_i/\sqrt d$ are orthonormal.
Projecting $A$ to the subspace $\operatorname{span}\{ e_i, Ae_i/\sqrt{d}\}$ gives
\[
  C_i= \begin{pmatrix} e_i^\top \\ e_i^\top A / \sqrt{d} \end{pmatrix} A \begin{pmatrix} e_i & A e_i / \sqrt{d} \end{pmatrix} = 
  \begin{pmatrix}
    0&\sqrt d\\
    \sqrt d&(A^3)_{ii}/d
  \end{pmatrix}.
\]
Since $A\preceq\rho I$, we have $C_i\preceq\rho I_2$. 
The determinant of $\rho I_2-C_i\succeq0$ gives $\rho(\rho-(A^3)_{ii}/d)\geq d$, proving \eqref{eq:triangle-compression}. 
This gives the bound on $\tau$ as $u$ is a unit vector with $\sum_i u_i^2 = 1$.
\end{proof}

Now, we are ready to derive a lower bound on the fourth-moment $m_4$.

\begin{lemma}[Lower bound on the fourth moment]
\label{lem:fourth-moment-lower}
For every signed adjacency matrix $A$ of a $d$-regular graph and an extremal unit eigenvector $u$ with $Au = \rho u$, it holds that
\begin{equation}
  16m_4\geq
  \rho^4+\left(2d-1+\frac1{4d}\right)\rho^2
  +2d^2-\frac{5d}{4}+\frac{d^3}{\rho^2}.
  \label{eq:fourth-moment-lower}
\end{equation}
\end{lemma}

\begin{proof}
Set
\[
  f(s)=-2\rho s+\frac{(s+\rho/2)^2}{d}.
\]
The minimum of $f$ occurs at $d\rho-\rho/2$. Since $\rho\leq d$, we have $d \rho - d^2/\rho < d\rho-\rho/2$. 
Thus, $f$ is decreasing on $(-\infty, d\rho - d^2/\rho]$, and \Cref{lem:diagonal-third-moment} gives
\begin{equation}
  f(\tau)\geq f\left( d \rho - \frac{d^2}{\rho} \right)
  = \frac{d^3}{\rho^2} - d +\Big(1-d+\frac{1}{4d}\Big) \rho^2.
  \label{eq:tau-quadratic-lower-bound}
\end{equation}
Using \Cref{lem:edgewise-decomposition} to lower bound \eqref{eq:exact-fourth-scalar}, we obtain
\[
	16m_4 \geq \rho^4 + (3d-2)\rho^2 + 2d^2 - \frac{d}{4} - 2 \rho \tau + \frac{(\tau + \rho/2)^2}{d}.
\]
Then, apply \eqref{eq:tau-quadratic-lower-bound} to obtain the final bound \eqref{eq:fourth-moment-lower}.
\end{proof}

The correction $d^2/\rho$ in the local triangle bound contributes $d^3/\rho^2$ to \eqref{eq:fourth-moment-lower}. 
Retaining this term is what improves the final constant.

\section{Two-sided bounds for graph signings}
\label{sec:signing-bounds}

Recall that \Cref{lem:mixed-root-signing} ensures the existence of a signing satisfying the mixed-root condition.
We now combine the moment estimates with this condition.

\subsection{General regular graphs}

We first derive the second moment bound using \Cref{cor:pair-recovery}.

\begin{corollary}[Second-moment bound]
\label{cor:sqrt-seven}
Let $A$ be the signed adjacency matrix of a signing of a finite $d$-regular graph ($d\geq2$) satisfying $r_A^2\leq2(d-1)$. Then, $A$ obeys
\[
  \rho(A)\leq\sqrt{7d-8}<\sqrt7\sqrt{d-1}.
\]
\end{corollary}

\begin{proof}
Since $\Delta_A=dI$ for $d$-regular graph, \eqref{eq:matrix-resolvent} gives $\rho(A)^2+d\leq4r_A^2\leq8(d-1)$.
\end{proof}

The fourth moment gives the improved bound in \Cref{thm:main-signing}.

\begin{proof}[Proof of \Cref{thm:main-signing}]
Fix a signing satisfying $r_A^2\leq2(d-1)$. 
Combining \eqref{eq:fourth-moment-upper} and \eqref{eq:fourth-moment-lower} yields
\begin{equation}
  \rho^4+\left(2d-1+\frac1{4d}\right)\rho^2
  +2d^2-\frac{5d}{4}+\frac{d^3}{\rho^2}
  \leq8(\rho^2+d)(d-1).
  \label{eq:combined-fourth-moment}
\end{equation}
Let $y=\rho^2/(d-1)>0$, subtract the right side and divide by $(d-1)^2$ to obtain
\begin{align}
  0\geq{}&\frac{(y+1)(y^2-7y+1)}{y}
  +\frac{y-\tfrac{21}{4}+3/y}{d-1}+\frac{\tfrac34+3/y}{(d-1)^2}
  +\frac{y}{4d(d-1)}+\frac1{(d-1)^3y}.
  \label{eq:factored-obstruction}
\end{align}
For
\[
  \frac{\rho^2}{d-1} = y\geq y_0:=\frac{7+3\sqrt5}{2}
       =\left(\frac{3+\sqrt5}{2}\right)^2,
\]
the first term is nonnegative and every remaining term is positive; in particular $y_0>21/4$. 
This contradicts \eqref{eq:factored-obstruction}, so
\[
  \rho(A)<\frac{3+\sqrt5}{2}\sqrt{d-1}.
\]
The signing was arbitrary among those satisfying the mixed-root condition.
There are finitely many such signings, and the set is nonempty by \Cref{lem:mixed-root-signing}. Its maximum therefore satisfies the same strict bound.
\end{proof}

\begin{remark}
\label{rem:asymptotic-constant}
Let us try to get some intuition about how does the constant $(3+\sqrt{5})/2$ appear.
Setting $\rho=c\sqrt d$ in \eqref{eq:combined-fourth-moment} and retaining the leading term as $d \to \infty$ gives
\[
  c^4+2c^2+2+c^{-2}\leq8(c^2+1).
\]
After multiplication by $c^2$, the difference factors as
\[
  c^6-6c^4-6c^2+1=(c^2+1)(c^4-7c^2+1),
\]
whose larger positive threshold is $c=(3+\sqrt5)/2$. Equation~\eqref{eq:factored-obstruction} verifies the normalization by $d-1$ for every degree. 
\end{remark}

\subsection{Graphs without triangles or four-cycles}

The exact third-moment and fourth-moment formulas become simpler when short cycles are absent. 
The following proof gives both bounds in \Cref{cor:short-cycles}.

\begin{proof}[Proof of \Cref{cor:short-cycles}]
As before, replace $A$ by $-A$ if necessary and choose a unit vector $u$ with $Au=\rho u$.

If the graph is triangle-free, then $(A^2)_{ij}=0$ on edges and $(A^3)_{ii}=0$ at every vertex. 
Also, $(A^4)_{ii}=\sum_j(A^2)_{ij}^2\geq d^2$. 
Hence \eqref{eq:exact-fourth-scalar} gives
\[
  16m_4\geq\rho^4+(3d-2)\rho^2+2d^2.
\]
Combining this with \eqref{eq:fourth-moment-upper}, we obtain
\begin{equation*}
  \rho^4+(6-5d)\rho^2-6d^2+8d\leq0.
\end{equation*}
As a quadratic in $\rho^2$, its constant term is negative for $d\geq2$, so $\rho^2$ must be at most the unique positive root, which proves \eqref{eq:triangle-free-bound}. 
Evaluating the left side at $\rho^2=6(d-1)$ gives $2d>0$, proving the strict inequality.

% If the girth is at least five, every closed walk of length four is backtracking, and its sign product is $+1$. 
% From each vertex, there are $d^2$ walks that return after two steps and then backtrack again, and $d(d-1)$ walks that reach distance two and retrace their steps.
% Thus $(A^4)_{ii}=2d^2-d$, and the exact fourth moment is
% \[
%   16m_4=\rho^4+(3d-2)\rho^2+3d^2-d.
% \]
% Comparison with \eqref{eq:fourth-moment-upper} yields
% \begin{equation*}
%   \rho^4+(6-5d)\rho^2-5d^2+7d\leq0.
% \end{equation*}
% The positive root gives the first inequality in \eqref{eq:girth-five-bound}. 
% For the second, write $y = \rho^2/(d-1)$. The left side becomes
% \[
%   (d-1)^2(y^2-5y-5)+(d-1)(y-3)+2.
% \]
% At $y=(5+3\sqrt5)/2$, the first term is zero and the others are positive. 
% The positive root is therefore strictly smaller than $((5+3\sqrt5)/2)(d-1)$.
% Finally, \Cref{lem:mixed-root-signing} gives a signing satisfying the condition in each graph class.

Then, we consider graphs with girth at least five.
Since there is no cycle of length $4$, every closed walk of length four is backtracking and its sign product is $+1$. 
From each vertex, there are $d^2$ walks that return after two steps and then backtrack again, and $d(d-1)$ walks that reach distance two and retrace their steps.
Thus, 
\[
  (A^4)_{ii}=d^2+d(d-1)=2d^2-d, \quad \forall i.
\]
Again, the graph is triangle-free, $(A^2)_{ij}=0$ on edges $\{i,j\} \in E$ and $(A^3)_{ii} = 0$ for all $i \in V$. 
Hence, it follows from \Cref{lem:fourth-moment-matrix} and \eqref{eq:exact-fourth-scalar} that
\[
  m_3=\frac{\rho^3+(2d-1)\rho}{8},
  \qquad
  m_4=\frac{\rho^4+(3d-2)\rho^2+3d^2-d}{16}.
\]

Since the matrix weights $W_\lambda$ are supported on $[-r_A,r_A]\subseteq[-\sqrt{2(d-1)},\sqrt{2(d-1)}]$, the second block-matrix inequality in \Cref{prop:moment-transfer}, with $m=2$ and $\Hmom_0=I$, gives
\[
  \begin{pmatrix}
    2(d-1)I-\Hmom_2 & 2(d-1) \Hmom_1-\Hmom_3\\
    2(d-1) \Hmom_1-\Hmom_3 & 2(d-1) \Hmom_2-\Hmom_4
  \end{pmatrix}
  \succeq 0.
\]
For arbitrary $a,b\in\mathbb R$, evaluating its quadratic form on $\binom{au}{bu} \in \R^{2n}$, we obtain
\[
  a^2(2(d-1)-m_2)+2ab(2(d-1)m_1-m_3)+b^2(2(d-1)m_2-m_4) \geq 0.
\]
Since this holds for every $a,b\in\mathbb R$, it follows that
\[
  \begin{pmatrix}
    2(d-1)-m_2 & 2(d-1)m_1-m_3\\
    2(d-1)m_1-m_3 & 2(d-1)m_2-m_4
  \end{pmatrix}
  \succeq 0.
\]
% Multiplying its determinant by $64$ and substituting
% the moment formulas yields
% \[
%   0\leq
%   d(7d-8)(5d-7)
%   -(6d^2-9d+1)\rho^2.
% \]
Multiplying the nonnegative determinant by $64$ and substituting the moment formulas yields
\begin{align*}
  0
  &\leq (7d-8-\rho^2)
    \bigl[-\rho^4+(5d-6)\rho^2+d(5d-7)\bigr]
    -\rho^2(6d-7-\rho^2)^2\\
  &=d(7d-8)(5d-7)-(6d^2-9d+1)\rho^2.
\end{align*}
Since $6d^2-9d+1>0$ for $d\geq2$, this proves the first inequality in \eqref{eq:girth-five-bound}.
The second strict inequality in \eqref{eq:girth-five-bound} follows by observing that 
\[
  \frac{35}{6}(d-1)
  -
  \frac{d(7d-8)(5d-7)}{6d^2-9d+1}
  =
  \frac{9d^2+14d-35}{6(6d^2-9d+1)}
  >0
  \qquad(d\geq2).
\]
Finally, \Cref{lem:mixed-root-signing} guarantees the existence of signings achieving the desired bound.
\end{proof}

\section{The \texorpdfstring{$K_8$}{K8} obstruction}
\label{sec:lower-bound}

We prove the lower bound in \Cref{thm:lower} using a signing with only two positive edges. 
On $K_8$ with vertex set $\{0,1,\ldots,7\}$, assign sign $+1$ to $\{1,2\}$ and $\{1,3\}$ and sign $-1$ to every other edge. 
Its signed adjacency matrix is
\begin{equation}
  A=\begin{pmatrix}
   0&-1&-1&-1&-1&-1&-1&-1\\
  -1& 0& 1& 1&-1&-1&-1&-1\\
  -1& 1& 0&-1&-1&-1&-1&-1\\
  -1& 1&-1& 0&-1&-1&-1&-1\\
  -1&-1&-1&-1& 0&-1&-1&-1\\
  -1&-1&-1&-1&-1& 0&-1&-1\\
  -1&-1&-1&-1&-1&-1& 0&-1\\
  -1&-1&-1&-1&-1&-1&-1& 0
  \end{pmatrix}.
  \label{eq:k8-matrix}
\end{equation}

\subsection{The adjacency spectrum}

Use the three vertex classes $\{1\}$, $\{2,3\}$, and $\{0,4,5,6,7\}$. 
Vectors supported on either of the last two classes and summing to zero are eigenvectors with eigenvalue $1$. 
These form a five-dimensional subspace. 
On vectors constant on each class, the action is represented by
\[
  Q=\begin{pmatrix}
    0&2&-5\\
    1&-1&-5\\
    -1&-2&-4
  \end{pmatrix},\qquad
  \det(xI-Q)=(x-3)(x^2+8x+11).
\]
Therefore
\begin{equation}
  \det(xI-A)=(x-1)^5(x-3)(x^2+8x+11),\qquad
  \rho(A)=4+\sqrt5.
  \label{eq:k8-spectrum}
\end{equation}

\subsection{The mixed polynomial}

\begin{proposition}
\label{prop:k8-mixed-characteristic}
The matrix in \eqref{eq:k8-matrix} satisfies
\begin{equation}
  q_A(x)=x^4(x^4-14x^2+25).
  \label{eq:k8-mixed-characteristic}
\end{equation}
\end{proposition}

\begin{proof}
For any vertex cut, the biadjacency matrix has the form $B=-J+R$, where $J$ is an all-ones matrix and $R$ is supported in one row or one column: the two positive edges share vertex $1$. 
Thus, $\operatorname{rank}B\leq2$, and the full cut matrix has rank at most
four. Bipartite symmetry gives
\[
  \det(xI-M_\varepsilon)=x^8-a_\varepsilon x^6+b_\varepsilon x^4.
\]
The coefficient $a_\varepsilon$ counts cut edges, so its average is $28/2=14$.

The coefficient $b_\varepsilon$ is the sum of the $4\times4$ principal minors. 
Only a $2+2$ cut of four vertices contributes. 
Its determinant is the square of a $2\times2$ signed biadjacency determinant, equal to $4$ when the corresponding $4$-cycle is negative and $0$ otherwise.
A fixed $4$-cycle is cut alternately with probability $2/2^4=1/8$.
Thus each negative $4$-cycle contributes $1/2$ to $\E b_\varepsilon$.

A $4$-cycle is negative exactly when it contains one of the two positive edges, but not both. 
Each positive edge belongs to $2\binom62=30$ unoriented $4$-cycles. 
Exactly five cycles contain both positive edges, one for each choice of the fourth vertex.
The number of negative cycles is therefore $30+30-2\cdot5=50$, and $\E b_\varepsilon=25$. 
Averaging the characteristic polynomials and applying \eqref{eq:random-cut-characteristic} proves the formula.
\end{proof}

\begin{proof}[Proof of \Cref{thm:lower}]
The nonzero squared roots of \eqref{eq:k8-mixed-characteristic} solve
$z^2-14z+25=0$. Hence
\[
  r_A^2=7+2\sqrt6<12=2(d-1).
\]
Together with \eqref{eq:k8-spectrum}, this gives
\[
  \frac{\rho(A)}{\sqrt{d-1}}
  =\frac{4+\sqrt5}{\sqrt6}\approx2.545864,
\]
proving the lower bound in \eqref{eq:Cmix-bounds}. The upper bound
follows from \Cref{thm:main-signing}.
\end{proof}

\section{Concluding remarks}
\label{sec:conclusion}

In this paper, we developed a spectral recovery framework for the two-sided control of signed adjacency matrices in the Bilu--Linial problem. 
The framework uses positive matrix-valued measures and moment inequalities to derive spectral bounds for $A_\sigma$ from the roots of $\chi[A_\sigma,-A_\sigma]$. 
Applying these comparisons under the mixed-root condition $r_{A_\sigma}\leq\sqrt{2(d-1)}$, we proved that every finite $d$-regular graph admits a signing $\sigma$ such that
\[
  \rho(A_\sigma) < \frac{3+\sqrt5}{2}\sqrt{d-1}.
\]
This improves the coefficient $2\sqrt2$ in the bound of Ravichandran and Srivastava~\cite{RS21}.

The positive matrix-valued representation is the main structural ingredient of the proof. 
Its exact moment identities connect the mixed roots to the input matrices and provide a systematic way to strengthen spectral comparisons. 
Applications of this representation to graph signings and related spectral problems merit further investigation.

Moreover, we also show the limitations of our approach of improving the Bilu--Linial constant.
In particular, the $K_8$ example shows that the mixed-root condition alone does not force every signing satisfying it to obey the Ramanujan threshold. 
More precisely, \Cref{thm:main-signing} and the $K_8$ example give
\[
  \frac{4+\sqrt5}{\sqrt6} \leq C_{\mathrm{mix}} \leq \frac{3+\sqrt5}{2},
\]
where $C_{\mathrm{mix}}$ is the optimal universal conversion constant under the mixed-root condition. 
Higher moments may improve the upper bound, although their use would require controlling the additional terms in their expansions. 
Determining the exact value of $C_{\mathrm{mix}}$ remains a natural extremal problem. 
In light of the counterexample of Xu~\cite{Xu26a}, it is also interesting to investigate whether additional structural information allows our framework to recover the Ramanujan threshold on restricted graph classes.

Finally, our proof is existential: the signing-selection step relies on the result of Ravichandran and Srivastava stated in \Cref{lem:mixed-root-signing}. 
A natural algorithmic question is whether a signing satisfying the mixed-root condition can be found in polynomial time.

\section*{Statement on AI Usage}

The authors have been working on the Bilu-Linial conjecture for some time.
In August 2026, the authors began using GPT-5.6 Sol Pro to explore proof approaches, verify ideas, and search for counterexamples.
The crucial idea of using the matrix-valued probability measure to recover spectral information was proposed by GPT-5.6; the $K_8$ example and some of the technical arguments were also provided by GPT-5.6.
An initial draft of this work was developed through iterative discussions with GPT-5.6 and later with GPT-6 Astra Pro.
Then, the authors manually verified all the proofs, simplified some of them, and improved the exposition of the manuscript.
The authors take full responsibility for the correctness of the proofs and the content in the manuscript.

\section*{Acknowledgements}

This work was supported in part by National Key R\&D Program of China, Natural
Science Foundation of China, and Fuzhou University Talent Fund.

\bibliographystyle{plain}

\bibliography{references}

\clearpage

\appendix

\section{Proof of \Cref{lem:fourth-moment-matrix}}
\label{app:fourth-moment}

We first express $\Hmom_3$ and $\Hmom_4$ in terms of expectations
of the random-cut matrix $M_\varepsilon$. We then evaluate these
expectations by averaging over the independent vertex signs.

\subsection{Proof of \eqref{eq:fourth-moment-covariance}}
Recall that $t_2=\tr(M_\varepsilon^2)$, and put
\[
  p_\varepsilon(z)=\det(I-zM_\varepsilon),
  \qquad
  \overline p(z)=\E p_\varepsilon(z).
\]
All power series below are considered near $z=0$, where
$\overline p(z)\ne0$ and $I-zM_\varepsilon$ is invertible for
every choice of $\varepsilon$. There are only finitely many
such choices, so we may average the series term by term.
Setting $x=1/z$ in \eqref{eq:averaged-resolvent} and using
$\adj(B)=\det(B)B^{-1}$ gives
\begin{align*}
  \sum_{m\geq0}\Hmom_m z^m
  &=\frac{F_A(1/z)}{z}
    =\frac{\E\bigl[p_\varepsilon(z)
                    (I-zM_\varepsilon)^{-1}\bigr]}
           {\overline p(z)}\\
  &=I+\E\left[
      \frac{p_\varepsilon(z)}{\overline p(z)}
      \bigl((I-zM_\varepsilon)^{-1}-I\bigr)
    \right].
\end{align*}
The last equality uses $\E p_\varepsilon(z)=\overline p(z)$.
% Subtracting $I$ inside the expectation will allow us to compute
% the coefficients of $z^3$ and $z^4$ without expanding the
% determinant to fourth order.

Indeed, $D_\varepsilon M_\varepsilon D_\varepsilon=-M_\varepsilon$
implies $p_\varepsilon(-z)=p_\varepsilon(z)$.
Moreover, $\tr M_\varepsilon=0$, so the coefficient of $z^2$
in $p_\varepsilon(z)$ is
$\bigl((\tr M_\varepsilon)^2-\tr(M_\varepsilon^2)\bigr)/2=-t_2/2$.
Thus
\[
  p_\varepsilon(z)=1-\frac12t_2z^2+O(z^4),
  \qquad
  \overline p(z)=1-\frac12(\E t_2)z^2+O(z^4),
\]
and division yields
\[
  \frac{p_\varepsilon(z)}{\overline p(z)}
  =1-\frac12(t_2-\E t_2)z^2+O(z^4).
\]
We multiply this expansion by
\[
  (I-zM_\varepsilon)^{-1}-I
  =zM_\varepsilon+z^2M_\varepsilon^2
     +z^3M_\varepsilon^3+z^4M_\varepsilon^4+O(z^5).
\]
Since this second series has zero constant term, the
$O(z^4)$ remainder in the first series contributes only
terms of order at least five. Comparing the coefficients
of $z^3$ and $z^4$ therefore gives
\begin{align*}
  \Hmom_3
    &=\E M_\varepsilon^3
      -\frac12\E\bigl[(t_2-\E t_2)M_\varepsilon\bigr],\\
  \Hmom_4
    &=\E M_\varepsilon^4
      -\frac12\E\bigl[(t_2-\E t_2)M_\varepsilon^2\bigr].
\end{align*}
By the definition of covariance, these are exactly
\eqref{eq:fourth-moment-covariance} for $j=3,4$.

\subsection{Proof of \eqref{eq:exact-third-matrix} and \eqref{eq:exact-fourth-matrix}}

We now evaluate the expectations and covariance terms in \eqref{eq:fourth-moment-covariance} to derive the explicit third-moment identity \eqref{eq:exact-third-matrix} and the fourth-moment formula \eqref{eq:exact-fourth-matrix}.

\noindent\emph{Averaging over the vertex signs.}
For the remaining calculations, write
\[
  D=D_\varepsilon,\qquad C=DAD,\qquad
  D_j=\diag(A^j),\qquad T=(A\circ A)\circ A^2.
\]
Then $2M_\varepsilon=A-C$. For every fixed matrix $X$,
the $(i,j)$ entry of $DXD$ is $\varepsilon_i\varepsilon_jX_{ij}$.
Independence of the signs gives $\E(DXD)=\diag(X)$.
Since $D^2=I$, we also have $C^j=DA^jD$, and hence
\[
  \E C^j=D_j,\qquad \E C=0,\qquad \E C^2=dI.
\]
Here we used $\diag(A)=0$ and $\diag(A^2)=dI$.

We next compute the averages of $CAC$, $CA^2C$, and $CAC^2$.
To treat these products together, let $X,Y$ be fixed matrices.
Expanding the $(i,j)$ entry gives
\[
  \bigl(\E[DAD\,X\,DYD]\bigr)_{ij}
  =\sum_{h,\ell}a_{ih}X_{h\ell}Y_{\ell j}
       \E(\varepsilon_i\varepsilon_h\varepsilon_\ell\varepsilon_j).
\]
The expectation vanishes unless every index occurs an even
number of times. When $i\ne j$, the only possible pairings
are $(h,\ell)=(j,i)$ and $(h,\ell)=(i,j)$; the latter
contributes zero because $a_{ii}=0$.
When $i=j$, the surviving terms have
$h=\ell$. Consequently,
\[
  \E[DAD\,X\,DYD]
  =A\circ X^\top\circ Y
    +\diag\bigl(A\diag(X)Y\bigr).
\]
Apply this identity with $(X,Y)=(A,A),(A^2,A),(A,A^2)$,
respectively. Since $A$ is symmetric,
$\diag(A^2)=dI$, and $(A\circ A)\circ A=A$, we obtain
\begin{align*}
  \E(CAC)&=A,\\
  \E(CA^2C)&=T+\diag\bigl(A(dI)A\bigr)=T+d^2I,\\
  \E(CAC^2)&=T.
\end{align*}
Taking transposes in the last identity gives
$\E(C^2AC)=T$ as well, since $T$ is symmetric.

\medskip
\noindent\emph{The third and fourth powers.}
In the expansion
\[
  (A-C)^3
  =A^3-A^2C-ACA-CA^2+AC^2+CAC+C^2A-C^3,
\]
all terms containing exactly one copy of $C$ have zero
expectation, because $\E C=0$. The identities above therefore give
\[
  8\E M_\varepsilon^3
  =A^3+A(dI)+A+(dI)A-D_3
  =A^3+(2d+1)A-D_3.
\]

To compute the fourth power, we square
$4M_\varepsilon^2=A^2+C^2-AC-CA$.
% This groups the terms into the two squares and their
% cross terms, while preserving the order of the matrix factors.
For the first square,
\begin{align*}
  \E\bigl[(A^2+C^2)^2\bigr]
    &=A^4+A^2(\E C^2)+(\E C^2)A^2+\E C^4\\
    &=A^4+2dA^2+D_4.
\end{align*}
For the second square,
\begin{align*}
  \E\bigl[(AC+CA)^2\bigr]
    &=A\E(CAC)+A(\E C^2)A
        +\E(CA^2C)+\E(CAC)A\\
    &=(d+2)A^2+d^2I+T.
\end{align*}
In the cross terms, the four products containing exactly
one copy of $C$ again have zero expectation. The remaining
four products give
\begin{align*}
  &\E\bigl[(A^2+C^2)(AC+CA)
           +(AC+CA)(A^2+C^2)\bigr]\\
  &\qquad=\E(C^2AC)+(\E C^3)A
             +A(\E C^3)+\E(CAC^2)\\
  &\qquad=AD_3+D_3A+2T.
\end{align*}
Adding the two square terms and subtracting the cross terms yields
\[
  16\E M_\varepsilon^4
  =A^4+(3d+2)A^2+d^2I-T-AD_3-D_3A+D_4.
\]

\medskip
\noindent\emph{The covariance corrections.}
Taking the trace of $4M_\varepsilon^2=A^2+C^2-AC-CA$,
and using $\tr C^2=\tr A^2=nd$, gives
\[
  t_2=\frac{nd}{2}-\frac12\tr(AC).
\]
Since $\E C=0$, we have $\E\tr(AC)=0$ and $\E t_2=nd/2$.
Thus $t_2-\E t_2=-\tr(AC)/2$.

To average the resulting products, expand
\[
  \tr(AC)=\sum_{p,q}a_{pq}^2\varepsilon_p\varepsilon_q.
\]
For a fixed matrix $X$ and $i\ne j$, the $(i,j)$ entry of
$\E[\tr(AC)DXD]$ receives contributions only from
$(p,q)=(i,j)$ and $(j,i)$, and is therefore $2a_{ij}^2X_{ij}$.
The diagonal entries on both sides are zero, since
$\E\tr(AC)=0$ and $a_{ii}=0$.
Hence
\[
  \E\bigl[\tr(AC)\,DXD\bigr]=2(A\circ A)\circ X.
\]
Using $C=DAD$ and $C^2=DA^2D$, the cases $X=A$ and
$X=A^2$ give
\[
  \E\bigl[\tr(AC)C\bigr]=2A,
  \qquad
  \E\bigl[\tr(AC)C^2\bigr]=2T.
\]
The covariance with $M_\varepsilon=(A-C)/2$ is therefore
\begin{align*}
  \Cov(t_2,M_\varepsilon)
    &=-\frac14\E\bigl[\tr(AC)(A-C)\bigr]\\
    &=-\frac14(0-2A)=\frac A2.
\end{align*}
Similarly, using $4M_\varepsilon^2=A^2+C^2-AC-CA$,
\begin{align*}
  \Cov(t_2,M_\varepsilon^2)
    &=-\frac18\E\bigl[\tr(AC)(A^2+C^2-AC-CA)\bigr]\\
    &=-\frac18\bigl(0+2T-2A^2-2A^2\bigr)
      =\frac{2A^2-T}{4}.
\end{align*}
% The zero terms in these two calculations use
% $\E\tr(AC)=0$; the fixed factors $A$ can be taken outside
% the expectation without changing their order.

Finally, substitute these covariances and the computed
expectations into \eqref{eq:fourth-moment-covariance}. For the
third moment, the correction to $8\E M_\varepsilon^3$ is $-2A$,
so
\[
  8\Hmom_3=A^3+(2d-1)A-D_3.
\]
For the fourth moment, the correction to $16\E M_\varepsilon^4$
is $-4A^2+2T$, which gives
\[
  16\Hmom_4
  =A^4+(3d-2)A^2+d^2I+T-AD_3-D_3A+D_4.
\]
Recalling $D_j=\diag(A^j)$ and $T=(A\circ A)\circ A^2$
proves the two claimed matrix identities.

%\subsection{The remainder in square completion}

\end{document}